\documentclass[12pt,oneside]{article}
\usepackage{amsmath,amssymb,amsfonts,amsthm}
\usepackage{color}
\usepackage[all]{xy}
\usepackage{graphicx}
\usepackage{color}
\usepackage{makeidx}
\usepackage{multirow}
\usepackage{rotating}
\usepackage{pdfpages}
\usepackage{pdflscape}
\usepackage{tabularray}
\usepackage{multicol}

\allowdisplaybreaks

\newcommand{\T}{{\cal T}}
\newcommand{\C}{{\cal C}}

\newcommand{\Real}{\mathbb R}

\newcommand{\tm}{\T M}

\newcommand{\TM}{\mathcal T\hspace{-1pt}M}

\def\pa{\partial}
\def\paa{\dot{\partial}}

\numberwithin{equation}{section} 
\numberwithin{figure}{section} 

\theoremstyle{plain}
\newtheorem*{theorem*}{Theorem}
\newtheorem{theorem}{Theorem}[section]
\newtheorem{lemma}[theorem]{Lemma}
\newtheorem{proposition}[theorem]{Proposition}
\newtheorem{corollary}[theorem]{Corollary}
\newtheorem{remark}[theorem]{Remark}

\newtheorem{definition}[theorem]{Definition}

\newtheorem{example}{Example}
\newtheorem*{acknowledgement*}{Acknowledgement}

\numberwithin{equation}{section}

\newcommand\overcirc[1]{\raisebox{10pt}{\tiny{$\circ$}}{\kern-7.5pt}\mbox{$#1$}}
\newcommand\undersym[2]{\raisebox{-6pt}{$#2$}{\kern-5pt}\mbox{$#1$}}
\newcommand\overdiamond[1]{\raisebox{10pt}{\small$\star$}{\kern-7.5pt}\mbox{$#1$}}
\newcommand\overast[1]{\raisebox{10pt}{\small$\ast$}{\kern-7.5pt}\mbox{$#1$}}
\newcommand\overlind[1]{\raisebox{10pt}{\small$\overline{{\hspace{2pt}}\star}$}{\kern-7.5pt}\mbox{$#1$}}
\newcommand\overlinc[1]{\raisebox{10pt}{\small$\overline{{\hspace{2pt}}\circ}$}{\kern-7.5pt}\mbox{$#1$}}
\newcommand\overlina[1]{\raisebox{10pt}{\small$\overline{{\hspace{1pt}}\ast}$}{\kern-7.5pt}\mbox{$#1$}}

\newcommand\undersymm[2]{\raisebox{-7pt}{\tiny$#2$}{\kern-15pt}\mbox{$#1$}}

\begin{document}

\title{\bf Parallel one forms on special Finsler manifolds }
\author{S. G. Elgendi}
\date{}

\maketitle
\vspace{-1 cm}

\begin{center}
{Department of Mathematics, Faculty of Science,\\
Islamic University of Madinah, Madinah, Saudi Arabia}
\vspace{-8pt}
\end{center}

\begin{center}
selgendi@iu.edu.sa, salahelgendi@yahoo.com
\end{center}

\vspace{0.3cm}
\begin{abstract} 
In this paper, we investigate  the existence of parallel 1-forms on specific Finsler manifolds. We demonstrate that Landsberg manifolds admitting a parallel 1-form have a mean Berwald curvature of rank at most $n-2$. As a result,  Landsberg surfaces with parallel 1-forms are necessarily Berwaldian. 
We further establish that the metrizability freedom of the geodesic spray for Landsberg metrics with parallel 1-forms is at least $2$. We figure out that some special Finsler metrics do not admit a parallel 1-form. Specifically, no parallel 1-form is admitted for any Finsler metrics of non-vanishing scalar curvature, among them the projectively flat metrics with non-vanishing scalar curvature. Furthermore,  neither the general Berwald's metric nor  the non-Riemannian spherically symmetric metrics admit a parallel 1-form. Consequently, we observe that certain $(\alpha,\beta)$-metrics and generalized $(\alpha,\beta)$-metrics do not admit parallel 1-forms.  
\end{abstract}
 



\section{Introduction}
~\par

 Parallel 1-forms find diverse applications in both Finsler (or Riemannian) geometry and physics, particularly in general relativity.
In Finsler geometry, parallel 1-forms play a significant role. For instance, within the class of $(\alpha,\beta)$-metrics, if the 1-form $\beta$ is parallel with respect to the Levi-Civita connection of the Riemannian metric $\alpha$, then the Riemannian metric $\alpha$ and the $(\alpha,\beta)$-metric share the same geodesic spray. Consequently, the $(\alpha,\beta)$-metric becomes a Berwald metric. Additionally, if $\beta$ is parallel, the Levi-Civita connection and the Cartan connection of the $(\alpha,\beta)$-metric coincide (see \cite{Percell, Shibata}).
From an application standpoint, in general relativity, if a metric $g$ admits a parallel vector field and satisfies the Einstein equations, then the energy-momentum tensor vanishes (see \cite{Mahara}).

In Riemannian geometry, a vector field is parallel if and only if its associated 1-form is parallel. This equivalence stems from the metricity  of the Levi-Civita connection, which implies that the covariant derivative of the metric tensor vanishes. However, in Finsler geometry, the situation is more complex, especially when  Finsler connection  is not metrical.

\bigskip

In \cite{Kozma-Elgendi}, L. Kozma and S. G. Elgendi delved into the concept of parallel 1-forms on Finsler manifolds. Specifically, considering a Berwald connection attached to  a Finsler space $(M,F)$, a 1-form $\beta=b_i(x)y^i$ is termed horizontally parallel (or simply, parallel) if and only if the Berwald horizontal covariant derivative of $b_i$ vanishes, i.e.,  $b_{i|j}=0$. For a Finsler space $(M,F)$, they explored the connection between the metrizability freedom of the geodesic spray of the Finsler structure  $F$ and the existence of parallel 1-forms on $(M,F)$. Furthermore, they employed Finslerian tools to discuss the presence of parallel 1-forms on both Riemannian and Finslerian manifolds.

 \bigskip
 
In this paper, we investigate  the existence of   parallel  1-forms on certain special Finsler spaces. First, we consider the Landsberg spaces admitting   parallel 1-forms. If $(M,F)$ is a  Landsberg   metric  and   provides    a  parallel 1-form, then the rank of the mean Berwald   curvature is at most $n-2$. As by-product, a  Landsberg surface that admits a parallel 1-form is Berwaldian. Moreover, if $(M,F)$ is a Landsberg manifold whose  geodesic spray is $S$,   then the metrizability freedom of $S$ is  at least $2$. 

The Finsler metrics with scalar curvature are the second type of special Finsler manifolds that we address in the present study. We prove that there is no parallel 1-form exist on  Finsler manifolds with non-vanishing scalar curvature. We consequently infer that no parallel 1-form can be admitted for any projectively flat Finsler metrics of non-vanishing scalar curvature. It is not enough for a Finsler metric $F$ to provide   a parallel 1-form merely to have a vanishing scalar curvature. Consider  the following   projectively flat metric  with zero flag curvature,  which is investigated and provided by Shen \cite{Shen-paper} 
\begin{align*}
\nonumber F(x, y)= & \left\{1+\langle a, x\rangle+\frac{\langle a, y\rangle-|x|^2\langle a, y\rangle}{\sqrt{|y|^2-|x|^2|y|^2+\langle x, y\rangle^2}+\langle x, y\rangle}\right\} \\
& \times \frac{\left(\sqrt{|y|^2-|x|^2|y|^2+\langle x, y\rangle^2}+\langle x, y\rangle\right)^2}{\left(1-|x|^2\right)^2 \sqrt{|y|^2-|x|^2|y|^2+\langle x, y\rangle^2}}  
\end{align*}
  admits no   parallel 1-forms,  where $|\cdot|$ (resp. $\langle \cdot , \cdot \rangle$) refers to the standard Euclidean norm (resp. inner product) on $\Real^n$.  By the way, since this metric generalizes Berwald's metric \cite{Berwald}, then we call it the general Berwld's metric.

Finally, we turn our attention to one of the most significant and diverse classes in Finsler geometry: the class of spherically symmetric   metrics. This class has several applications in both Finsler geometry and physics. Now, let  $F=u\phi(r,s)$ be  spherically symmetric Finsler metric admitting a parallel 1-form, then we     its  geodesic spray  is characterized by   the following special formulae of the functions $P$ and $Q$:
 \begin{equation*}
P=P(r,s), \quad Q=\frac{ s^2 f'(r)}{2r^3f(r)}-\frac{sP }{r^2}+\frac{1}{2r^2}
\end{equation*}
where  $b_i=f(r) x_i$,  $f':=\frac{d f}{d r}$, and   $f(r)$ is a     smooth function  of $r$.  Then, a question arises, precisely,\textit{ is this spray   metrizable?} We show that this  spray    is only Riemann metrizable. That is,   there is no  non-Riemannian  spherically symmetric metric    provides  a parallel 1-form.
 
 \medskip

The class of spherically symmetric metrics is  an example  of generalized $(\alpha,\beta)$-metrics, while the general Berwald metricis an example of $(\alpha,\beta)$-metrics. In conclusion, we demonstrate the existence of $(\alpha,\beta)$-metrics and generalized $(\alpha,\beta)$-metrics that do not admit parallel 1-forms.

\section{Preliminaries}

Let $M$ be a smooth manifold of $n$ dimensions, and let its tangent bundle $(TM,\pi_M,M)$, and let $(\TM,\pi,M)$ be a subbundle of non-zero tangent vectors. We use $(x^i, y^i)$ to represent the induced coordinates of $TM$, where $(x^i)$ is the local coordinate of a base point   $x\in M$ and $(y^i)$ represents the    tangent vectors      $y\in T_xM$, where $T_xM$ is tangent space at $x$. The tangent structure $J$ of $T M$ is a vector 1-form   defined locally by $J = \frac{\partial}{\partial y^i} \otimes dx^i$, where $\otimes $ is the tensor product of $\frac{\partial}{\partial y^i}$ and $dx^i$. The  Liouville or canonical     vector field $\C$ is a   vector field on $TM$ and is defined by $\C=y^i\frac{\partial}{\partial y^i}$.

\medskip

A spray  is a vector field $S$ given on the tangent bundle   $ T M $  such that     $JS = \C$,  and   $[\C, S] = S$. It can be written  locally as in the following  expression:
\begin{equation}
  \label{eq:spray}
  S = y^j \frac{\partial}{\partial x^j} - 2G^j\frac{\partial}{\partial y^j},
\end{equation}
where the functions  $G^j=G^j(x,y)$   are called  the spray coefficients. These functions  are smooth   and $2$-homogeneous   in $y$.

\medskip

A non-linear connection is defined by an  $n$-dimensional distribution $H$ on $\tm$, which is the complement of the vertical distribution $V\tm$.  So, for each $z \in \tm$,  we have the following direct sum  
$$T_z(\tm) = H_z(\tm) \oplus V_z(\tm).$$  
Each spray $S$  can be associated by   a canonical non-linear connection with  a horizontal and vertical projectors given as follows 
\begin{equation}
  \label{projectors}
    h=\frac{1}{2}  (Id + [J,S]), \,\,\,\,            v=\frac{1}{2}(Id - [J,S])
\end{equation}
 Locally, the horizontal  projector  $h$ and the vertical projector $v$ are  expressed by the formulae 
$$h=\frac{\delta}{\delta x^k}\otimes dx^k, \quad\quad v=\frac{\partial}{\partial y^k}\otimes \delta y^k,$$
$$\frac{\delta}{\delta x^k}=\frac{\partial}{\partial x^k}-N^i_k(x,y)\frac{\partial}{\partial y^i},\quad \delta y^k=dy^k+N^k_i(x,y)dx^i, \quad N^h_i(x,y)=\frac{\partial G^h}{\partial y^i},$$
where $N^k_i$ are the components  of  the nonlinear connection.  

\medskip

Let $K$ be  a vector
  $k$-form on $M$, that is, 
  $K:(\mathfrak{X}(M))^k\longrightarrow \mathfrak{X}(M)$. Each vector $k$-form $K$ induces    graded
  derivations of the Grassmann algebra of $M$, namely   $i_K$ and $d_K$ as
  follows:
$$i_K\varphi=0, \,\,\,\, i_Kd\varphi=d\varphi\circ K,$$
$$d_K:=[i_K,d]=i_K\circ d-(-1)^{k-1}di_K,$$
where $\varphi\in C^\infty(M)$, $d\varphi$ represents the differential of $\varphi\in C^\infty(M)$ .
As a special case, for a vector field $\xi\in \mathfrak{X}(M)$, then we have the Lie derivative $\mathcal{L}_\xi$   with respect to $\xi$ and  the interior product $i_{\xi}$     by $\xi$.

\medskip

The Jacobi endomorphism (or, Riemann curvature)  \cite{Grifone_1972} is defined by
$$\Phi=v\circ [S,h]=R^i_j\frac{\partial}{\partial y^i}\otimes dx^j=\left(2\frac{\partial G^i}{\partial x^j}-S(N^i_j)-N^i_kN^k_j \right)\frac{\partial}{\partial y^i}\otimes dx^j.$$

The    curvature $R$ of $S$ is defined by
\begin{displaymath}
  R=\frac{1}{2}[h,h]=\frac{1}{2}R^h_{jk}\frac{\partial}{\partial
    y^h}\otimes dx^j \wedge dx^k, \qquad R^h_{jk} =
  \frac{\delta
    G^h_j}{\delta x^k} - \frac{\delta G^h_k}{\delta x^j}.
\end{displaymath}
 One can see that $R^h_i=R^h_{ij}y^j$. For more details, refer to \cite{Salah_ams_1}, for example.
 
 We adopt  the notations
$${\partial}_i:=\dfrac{\partial}{\partial x^i}, \quad \dot{\partial}_i:=\dfrac{\partial}{\partial y^i}, \quad  \delta_i:=\dfrac{\delta}{\delta x^i}=\dfrac{\partial}{\partial x^i}-G^j_i(x,y)\dfrac{\partial}{\partial y^j}.$$
 The Berwald connection's coefficients  $ G^h_{ij}$  \cite{shen-book1}  are given by    $ G^h_{ij}=\dfrac{\partial G^h_j}{\partial y^i}$.

\begin{definition}
 A pair $(M,F)$ is termed a Finsler manifold (or, Finsler space), wherein $M$ denotes a smooth $n$-dimensional manifold and $F: TM \to \mathbb{R}$  with    the   properties:
 \begin{description}
   \item[(a)] $F$ is strictly positive and   smooth  on $\T M$.
   \item[(b)]  $F$ is positively $1$-homogeneous  in   $y$.
    \item[(c)] The metric tensor $g_{ij}= \paa_i\paa_j E $ has rank $n$, where  $E:=\frac{1}{2}F^2$ is the energy function.
 \end{description}
The function $F$ is known as a Finsler function (or structure, or metric).
 \end{definition}
 
The  Berwald  curvature tensor     $\mathcal{G}$,     and the Landsberg curvature tensor $\mathcal{L}$ are defined, respectively,  by
\begin{equation}
\label{Berwald_curv.}
\mathcal{G}=G^h_{ijk} dx^i\otimes dx^j\otimes dx^k\otimes\paa_h
\end{equation}
\begin{equation}
\label{Landsberg_Tensor}
\mathcal{L}=L_{ijk} dx^i\otimes dx^j\otimes dx^k,
\end{equation} where $L_{ijk}=-\frac{1}{2}F G^h_{ijk}\paa_hF$, $G^h_{ijk}=\paa_kG^h_{ij}$ ,  see \cite{Shen-book}.
  The mean Berwald curvature $E_{j k}$ of a spray $S$   is defined by \cite[Def. 6.1.2]{shen-book1} as follows
$$E_{j k}=\frac{1}{2} G^i_{i j k}=\frac{1}{2} \frac{\partial^3 G^i}{\partial y^i \partial y^j \partial y^k}.$$
 \begin{definition}
A Berwald space is a Finsler space $(M,F)$ where  the components $G^{h}_{ijk}$  of  Berwald curvature tensor vanishes identically. Similarly, a {Landsberg} space is a Finsler space    $(M,F)$  in which the components $L_{jkh}$ of the  Landsberg curvature tensor  vanishes identically.
\end{definition}

\section{Parallel 1-forms on Landsberg  manifolds}

In \cite{Kozma-Elgendi}, L. Kozma and the author of this article investigated and studied the concept of parallel 1-forms on Riemannian and Finsler manifolds. They characterized the existence of parallel 1-forms in general. Here, in this section, we begin to study the presence of parallel 1-forms on some specific Finsler spaces of interest. Precisely, we start with Landsberg metrics. Let us provide the definition of a parallel 1-form.

\begin{definition} \cite{Kozma-Elgendi}
Let $\beta=b_i(x) y^i$ be a 1-form on a Finsler space $(M, F)$. Then, $\beta$ is said to be a horizontally parallel (or simply, parallel) 1-form with respect to the attached Berwald connection to $F$ if $b_{i \mid j}=0$, where the symbol $\mid$ denotes the Berwald horizontal covariant derivative.
\end{definition}
Considering a parallel 1-form $\beta=b_i(x) y^i$, according to \cite{MZ_ELQ}, we observe that $\beta$ is a holonomy-invariant function on the slit tangent bundle $ \T M $, meaning $d_h \beta=0$. As a compatibility condition, $\beta$ must satisfy the property $d_R \beta=0$. Furthermore, considering that $\beta$ is a function on $ \T M $ and homogeneous of degree 1 in $y$, we have $d_C \beta=\beta$. Summarizing these facts, we conclude that the existence of a parallel 1-form $\beta=b_i(x) y^i$ on a Finsler space $(M, F)$ can be characterized by the following system:
$$d_h\beta=0,  \quad    d_\C\beta=\beta,$$
and additionally, the compatibility condition  $d_R\beta=0$,  where $d_h\beta(X)=hX(\beta)$ for all $X\in \mathfrak{X}(\T M)$, $d_R\beta=R(\beta)$, and $d_\C\beta=\C(\beta)$. 

\medskip

 The following lemma is required  for subsequent use.

\begin{lemma}\label{Lem:m_i_Not_0}
Consider a 1-form $\beta=b_i(x)y^i$    on a Finsler space $(M,F)$, then the  covector defined by 
$$m_j=b_j-\frac{\beta}{F}\ell_j,$$
where $\ell_j:=\paa_iF$, is non-vanishing on $\T M$, that is, $m_j\neq0$. 
\end{lemma}
\begin{proof}
The proof is proceeded by contradiction. Let $m_j=0$, then we have
$$b_j-\frac{\beta}{F}\ell_j=0.$$
Taking the derivative with respect to $y^k$, we get
$$-\frac{1}{F}b_k\ell_j-\frac{\beta}{F} \ell_{jk}+\frac{\beta}{F^2}\ell_j\ell_k=0.$$
Substituting by $b_j=\frac{\beta}{F}\ell_j$ implies
$$-\frac{\beta}{F^2}\ell_k\ell_j-\frac{\beta}{F} \ell_{jk}+\frac{\beta}{F^2}\ell_j\ell_k=0.$$
Therefore, we conclude that the angular metric $h_{ij}=F\ell_{ij}=0$, where $\ell_{ij}=\paa_j\paa_iF$.  Contracting $h_{ij}=g_{ij}-\ell_i\ell_j=0$ by the components of the inverse metric tensor yields  $g^{ij}h_{ij}=g^{ij}(g_{ij}-\ell_i\ell_j)=n-1 =0$, that is, $n=1$ which is a contradiction.
\end{proof}

\begin{theorem}
Let $(M,F)$ be a  Landsberg space   admitting  a   parallel 1-form, then the rank of the mean Berwald curvature is at most $n-2$.
\end{theorem}

\begin{proof}
Let $(M,F)$ be a be a Landsberg space admitting a  parallel 1-form $\beta=b_iy^i$. Then,  by \cite{Kozma-Elgendi}, we have 
$$\ell_hG^h_{ijk}=0, \quad b_hG^h_{ijk}=0.$$
Since for a Landsberg manifold, the Berwald curvature satisfies the property that $G_{hijk}:=g_{\ell h}G^\ell_{ijk}$ is completely symmetric, then the property $b_hG^h_{ijk}=0$ implies that
 $$ b^iG^h_{ijk}=0.$$
As the mean Berwald curvature  $E_{ij}=\frac{1}{2} G^h_{hij}$, then we get
$$b^iE_{ij}=0.$$ 
 
Now, we show that $y^i$ and $b^i$ are independent. Assume the combination
$$\mu y^i+\lambda b^i=0 $$
for some functions $\mu$ and $\lambda$ on $\T M$. 
By contracting the above equation by $h_{ij}$, we have
$$\lambda (b_j-\frac{\beta}{F}\ell_j)=0.$$
By making use of Lemma \ref{Lem:m_i_Not_0},  we  obtain that $\lambda=0$ and hence $\mu=0$. That is, $y^i$ and $b^i$ are independent and thus the rank of  $E_{ij}$ is at most $n-2$.
\end{proof}

By utilizing the above theorem together with \cite[Theorem A]{crampin2}, we get  the following result.
\begin{theorem}
A Landsberg surface that admits a parallel 1-form is Berwaldian.
\end{theorem}
The following result is a generalized version of \cite[Proposition 3.8]{Kozma-Elgendi}.
\begin{proposition}\label{F_beta_indep.}
Consider  a 1-form $\beta=b_i(x)y^i$   on  a Finsler manifold  $(M,F)$. Then,  the Finsler functions $F$ and $\overline{F}=F\varphi(\textbf{s}), \ \textbf{s}:=\frac{\beta}{F}$, defined on $M$,  are locally functionally independent.  Where $\varphi$ is an appropriate  non-constant, positive, and smooth   function on $\Real$.
\end{proposition}
\begin{proof}
Let   $F$ and $\overline{F}=\varphi(\textbf{s})F$ be two  functionally dependent functions. Then, the   $2$-form $d\overline{F}\wedge d F$ vanishes, that is we have 
$$d\overline{F}\wedge d F=\frac{\partial \varphi}{\partial \textbf{s}} \frac{1}{F} d\beta\wedge d F=0.$$
 Since $\varphi$ is    non-constant,  then
 $\frac{\partial \varphi}{\partial \textbf{s}}\neq 0$,  and thus $d\beta\wedge d {F}=0$. Moreover,  we have 
\begin{eqnarray*}
0=d\beta\wedge d {F} = \pa_i \beta \ \pa_j {F} \ dx^i\wedge dx^j
+  \paa_i \beta \ \paa_j {F} \ dy^i\wedge dy^j
+\left(\pa_i \beta \ \paa_j {F}-\pa_i {F} \ \paa_j \beta \right) dx^i\wedge dy^j.
\end{eqnarray*}
The above equation holds if and only if    each term   vanishes. In particular,  the combination or the term 
 $ \paa_i \beta \ \paa_j {F} \ dy^i\wedge dy^j$. This combination  vanishes if and only if  $ \paa_i \beta \ \paa_j {F}  $ is symmetric in $i$ and $j$, hence we have
 $$ \paa_i \beta \ \paa_j {F}  -\paa_j \beta \ \paa_i {F} ={\ell}_{i}b_j-{\ell}_jb_i=0, \quad {\ell}_i:=\paa_i {F}.$$
Then contraction the above equation  by $y^i$ implies   ${F} b_j-\beta{\ell}_j=0$. Then, taking the derivative with respect to $y^k$, keeping  the fact  that ${\ell}_{i}b_j={\ell}_jb_i$ in mind, we get
 $${F} \beta {h}_{jk}=0$$
  where ${h}_{jk}$ is the angular metric attached to  the Finsler structure ${F}$. Since all of the objects    ${h}_{jk}$,   ${F}$ and $\beta$ are non-vanishing,  then    a contradiction is attained.   
\end{proof}

\begin{theorem}\label{Th:Lands_parallel_form}
 Consider a Landsberg space $(M,F)$ with the geodesic spray $S$ and admitting a parallel 1-form.  Then, $S$ has  metrizability freedom at least $2$. 
\end{theorem}

\begin{proof}
Let $(M,F)$ be a be a Landsberg space and  admits  a horizontally parallel 1-form $\beta=b_iy^i$. Then we have 
$$\ell_hG^h_{ijk}=0, \quad b_hG^h_{ijk}=0.$$
Which read that $$(\ell_h+b_h)G^h_{ijk}=0.$$
If one considers $$\overline{\ell}_r:=\ell_r+b_r. $$
Contracting the above equation by $y^r$ implies  $\overline{F}=F+\beta$. That is,  we get a Randers change for $F$ by the parallel form $\beta$, moreover,  $\overline{F}$ has the same geodesic spray as $F$ because  $\beta$ is parallel, that is $\overline{G}^r_{ijk}=G^r_{ijk}$. So, we get a Randers metric $\overline{F}$ in which $\overline{\ell}_r\overline{G}^r_{ijk}=0$, which means that $\overline{F}$ is Landsberg. Moreover, by Proposition \ref{F_beta_indep.}, $F$ and $\overline{F}$ are functionally independent.  Then, by \cite{Mu-Elgendi}, the result follows.
\end{proof}

\section{Parallel 1-forms on Finsler metrics of scalar curvature}

Flag curvature is an important  object in Finsler geometry, analogous to sectional curvature in Riemannian geometry. Finsler metrics with scalar flag curvature are of particular interest.
In this section, we explore the existence of parallel 1-forms on Finsler metrics of scalar flag curvature. Let's present the following definition:
\begin{definition}\cite{shen-book1}
A Finsler space   $(M,F)$ is said to have  scalar flag curvature (or simply, scalar   curvature) if its Riemann curvature (Jacobi endomorphism) is given in the form
\begin{equation}
\label{Eq:Scalar_curv.}
R^h_i=K  (F^2\delta^h_i-y_i y^h),
\end{equation}
where $K(x,y)$ is a smooth function on $\T M$.
\end{definition}
Now, we have the   following  result.
\begin{theorem}
All Finsler spaces of non vanishing scalar curvature     admits no  parallel 1-forms.
\end{theorem}
 \begin{proof}
 Let  $(M,F)$ be a Finsler space admitting a parallel 1-form $\beta$, then we have
 $$d_R\beta=0 \Longrightarrow R(\beta)=R^h_{jk}\paa_b\beta=R^h_{jk}b_h=0\Longrightarrow y^k R^h_{jk}b_h=R^h_jb_h=0.$$
 Now, by making use of \eqref{Eq:Scalar_curv.}, we obtain  
 $$R^h_{j}b_h=KF^2\left(b_j-\frac{\beta}{F^2} y_j\right)=KF^2m_j=0.$$
 Since both $K$ and $F$ are non-zero, then $m_j=0$. But by Lemma \ref{Lem:m_i_Not_0},  we get a contradiction and consequently, the Finsler manifold $(M,F)$ does not provide a parallel 1-form.
 \end{proof}
 
 Since every  projectively flat Finsler metric is of scalar curvature, then we have the following corollary.
 \begin{corollary}
A projectively flat Finsler metric with non-zero scalar curvature does not admit parallel 1-forms.
 \end{corollary}

Two examples of projectively flat metrics with  vanishing curvature are presented below. While the second example does not provide  a parallel 1-form, the first one does. This illustrates that the presence  of a parallel 1-form does not necessitate the vanishing of the flag curvature. 
\begin{example}
Let $|\cdot|$ and $\langle \cdot, \cdot \rangle$ denote the Euclidean norm and inner product on $\mathbb{R}^n$, respectively. Consider the Finsler metric $F(x,y)$ on the unit ball $\mathbb{B}^n$ defined by:
\begin{equation*}
 F(x,y)=\frac{\sqrt{1-|a|^2}}{(1+\langle a,x\rangle)^2}\sqrt{|y|^2-\frac{2\langle a,y \rangle\langle x,y \rangle }{1+\langle a,x\rangle}-\frac{(1-|x|^2)\langle a,y \rangle^2 }{1+\langle a,x\rangle}},
\end{equation*}
where $y\in T_x B^n=\mathbb{R}^n$, $a=(a_1,a_2,...,a_n)\in \mathbb{R}^n$ is a   fixed vector with $|a|<1$.  
 The spray coefficients $G^i$ of the geodesic spray of $F$ are  
 $$G^i=-\frac{\langle a,y \rangle}{1+\langle a,x \rangle}y^i. $$
By \cite{Kozma-Elgendi},  $F$ is a projectively flat metric with zero curvature, and admit  a parallel one  form $\beta=b_i(x)y^i$ defined  by the component $b_i$ as follows 
$$b_1(x)=\frac{c+c_\mu x^\mu}{(1+\langle a,x\rangle)^2},\quad b_\mu(x)=\frac{a_\mu b_1}{a_1}-\frac{c_\mu(1+\langle a,x\rangle)}{a_1(1+\langle a,x\rangle)^2},$$
where  $\mu =2,...,n$.  
\end{example}

\begin{example} 
Consider the class of  projectively flat metrics with zero flag curvature studied by Shen \cite{Shen-paper} and given by 
\begin{align}\label{Eq:Shen_class}
\nonumber F(x, y)= & \left\{1+\langle a, x\rangle+\frac{\langle a, y\rangle-|x|^2\langle a, y\rangle}{\sqrt{|y|^2-|x|^2|y|^2+\langle x, y\rangle^2}+\langle x, y\rangle}\right\} \\
& \times \frac{\left(\sqrt{|y|^2- |x|^2|y|^2+\langle x, y\rangle^2}+\langle x, y\rangle\right)^2}{\left(1-|x|^2\right)^2 \sqrt{|y|^2-\left(|x|^2|y|^2-\langle x, y\rangle^2\right)}} ,
\end{align}
with the geodesic spray   coefficients 
\begin{equation*}
  G^i=\mathcal{P}y^i=\frac{\sqrt{|y|^2-|x|^2|y|^2+\langle x,y\rangle^2}+\langle x,y\rangle}{1-|x|^2} y^i,
\end{equation*}
where $\mathcal{P}$ is the projective factor   given by
$$\mathcal{P}=\frac{\sqrt{|y|^2-|x|^2 |y|^2+\langle x,y\rangle^2}+\langle x,y\rangle}{1-|x|^2}.$$
This metric does not provide a parallel 1-form as it is shown below.
 \end{example}

\begin{remark} Since by choosing  $a=0$ in \eqref{Eq:Shen_class}, we get  the Berwald's metric  \cite{Berwald}
 \begin{equation}\label{Berwal's_Metric}
  F=\frac{(\sqrt{|y|^2-|x|^2 |y|^2+\langle x,y\rangle^2}+\langle x,y\rangle)^2}{(1-|x|^2)^2\sqrt{|y|^2-|x|^2 |y|^2+\langle x,y\rangle^2}},
\end{equation}
we call the class \eqref{Eq:Shen_class} the \textit{general Berwald's metric}.
\end{remark}

\begin{proposition}
The Berwald curvature $G^h_{ijk}$ of the general Berwald's metric \eqref{Eq:Shen_class} is given by
\begin{align}
\label{Eq:Berwald_curv.}
\nonumber G^h_{ijk}&=\frac{1}{L}\frac{1}{1-|x|^2}\left(  \delta_{ij} \delta^h_k+\delta_{jk} \delta^h_i+\delta_{ki} \delta^h_j\right)-\frac{1}{L^3}  \frac{1}{(1-|x|^2)^2}(y_i\delta_{jk}+y_j\delta_{ki}+y_k\delta_{ij})y^h\\
\nonumber &-\frac{1}{L^3}  \frac{1}{(1-|x|^2)^2}(y_iy_j\delta^h_k+y_jy_k\delta^h_i+y_ky_i\delta^h_j)-\frac{1}{L^3}  \frac{\langle x,y\rangle}{(1-|x|^2)^3}(x_i\delta_{jk}+x_j\delta_{ki}+x_k\delta_{ij})y^h\\
\nonumber &+\left(-\frac{1}{L^3}\frac{\langle x,y\rangle^2 }{(1-|x|^2)^4}+\frac{1}{L}\frac{1}{(1-|x|^2)^2}\right)(x_ix_j\delta^h_k+x_jx_k\delta^h_i+x_kx_i\delta^h_j)\\
&-\frac{1}{L^3}  \frac{\langle x,y\rangle}{(1-|x|^2)^3}\left( (x_iy_j+x_jy_i) \delta^h_k+(x_jy_k+x_ky_j) \delta^h_i+(x_ky_i+x_iy_k) \delta^h_j\right)\\
\nonumber &+\left(-\frac{3}{L^3}  \frac{\langle x,y\rangle}{(1-|x|^2)^4}+\frac{3}{L^5}  \frac{\langle x,y\rangle^3}{(1-|x|^2)^6}\right)  x_ix_jx_ky^h+\frac{3}{L^5}  \frac{1}{(1-|x|^2)^3}y_iy_jy_ky^h\\
\nonumber &+\frac{3}{L^5} \frac{\langle x,y\rangle}{(1-|x|^2)^4}\left(y_iy_jx_k+y_jy_kx_i+y_ky_ix_j\right)y^h\\
\nonumber &+\left(\frac{3}{L^5}  \frac{\langle x,y\rangle^2}{(1-|x|^2)^5}-\frac{1}{L^3}  \frac{1}{(1-|x|^2)^3}  \right)\left(y_ix_jx_k+y_jx_kx_i+y_kx_ix_j\right)y^h.
\end{align}
\end{proposition}
\begin{proof}
For simplicity, let's write the projective factor  $\mathcal{P}$ in the form
$$\mathcal{P}=L+\frac{\langle x,y\rangle}{1-|x|^2}, \quad L:=\frac{\sqrt{|y|^2-|x|^2 |y|^2+\langle x,y\rangle^2}}{1-|x|^2}.$$
Now, taking the derivative of the projective factor $\mathcal{P}$ with respect to $y^i$, we have
$$\mathcal{P}_{i}:=\paa_i \mathcal{P}=\frac{1}{L}\left( \frac{y_i}{1-|x|^2}+\frac{\langle x,y\rangle x_i}{(1-|x|^2)^2}\right)+\frac{x_i}{1-|x|^2}.$$
Similarly, taking the derivative of  $\mathcal{P}_i$ with respect to $y^j$, we get
$$\mathcal{P}_{ij}:=\paa_i\mathcal{P}_j=-\frac{1}{L^3}\left( \frac{y_iy_j}{(1-|x|^2)^2}+\frac{\langle x,y\rangle(x_iy_j+x_jy_i)}{(1-|x|^2)^3}+\frac{\langle x,y\rangle^2 x_ix_j}{(1-|x|^2)^4}\right)+\frac{1}{L}\left( \frac{\delta_{ij}}{1-|x|^2}+\frac{x_ix_j}{(1-|x|^2)^2}\right).$$
Furthermore, taking the derivative of  $\mathcal{P}_{ij}$ with respect to $y^k$, we obtain the   formula
\begin{align*}
\mathcal{P}_{ijk}&=\paa_k \mathcal{P}_{ij}=\frac{3}{L^5} \frac{\langle x,y\rangle}{(1-|x|^2)^4}\left(y_iy_jx_k+y_jy_kx_k+y_ky_ix_j\right)\\
&+\left(\frac{3}{L^5}  \frac{\langle x,y\rangle^2}{(1-|x|^2)^5}-\frac{1}{L^3}  \frac{1}{(1-|x|^2)^3}  \right)\left(y_ix_jx_k+y_jx_kx_k+y_kx_ix_j\right)\\
&-\frac{1}{L^3}  \frac{1}{(1-|x|^2)^2}(y_i\delta_{jk}+y_j\delta_{ki}+y_k\delta_{ij})-\frac{1}{L^3}  \frac{\langle x,y\rangle}{(1-|x|^2)^3}(x_i\delta_{jk}+x_j\delta_{ki}+x_k\delta_{ij})\\
&+\left(-\frac{3}{L^3}  \frac{\langle x,y\rangle}{(1-|x|^2)^4}+\frac{3}{L^5}  \frac{\langle x,y\rangle^3}{(1-|x|^2)^6}\right)  x_ix_jx_k+\frac{3}{L^5}  \frac{1}{(1-|x|^2)^3}y_iy_jy_k.
\end{align*}
By substitution by the above formulae of $\mathcal{P}_{ij}$ and $\mathcal{P}_{ijk}$ in to the Berwald curvature 
$$G^h_{ijk}=\mathcal{P}_{ijk}y^h+\mathcal{P}_{ij}\delta^h_k+\mathcal{P}_{jk}\delta^h_i+\mathcal{P}_{ki}\delta^h_j$$
the result follows.
\end{proof}

\begin{theorem}\label{Th:Berwald's_Metric}
The general Berwald's metric    \eqref{Eq:Shen_class}    does not admit a parallel 1-form 
\end{theorem}

\begin{proof}
Let $\beta=b_iy^i$ be a parallel 1-form. Then, we have the condition \cite{Kozma-Elgendi} $G^h_{ijk}b_h=0$. By \eqref{Eq:Berwald_curv.}, we have
\begin{align*}
0&=G^h_{ijk}b_h\\&=\frac{1}{L}\frac{1}{1-|x|^2}\left(  \delta_{ij} b_k+\delta_{jk} b_i+\delta_{ki} b_j\right)-\frac{1}{L^3}  \frac{1}{(1-|x|^2)^2}(y_i\delta_{jk}+y_j\delta_{ki}+y_k\delta_{ij})\beta\\
&-\frac{1}{L^3}  \frac{1}{(1-|x|^2)^2}(y_iy_jb_k+y_jy_kb_i+y_ky_ib_j)-\frac{1}{L^3}  \frac{\langle x,y\rangle}{(1-|x|^2)^3}(x_i\delta_{jk}+x_j\delta_{ki}+x_k\delta_{ij})\beta\\
&+\left(-\frac{1}{L^3}\frac{\langle x,y\rangle^2 }{(1-|x|^2)^4}+\frac{1}{L}\frac{1}{(1-|x|^2)^2}\right)(x_ix_jb_k+x_jx_kb_i+x_kx_ib_j)\\
&-\frac{1}{L^3}  \frac{\langle x,y\rangle}{(1-|x|^2)^4}\left( (x_iy_j+x_jy_i) b_k+(x_jy_k+x_ky_j) b_i+(x_ky_i+x_iy_k) b_j\right)\\
&+\left(-\frac{3}{L^3}  \frac{\langle x,y\rangle}{(1-|x|^2)^4}+\frac{3}{L^5}  \frac{\langle x,y\rangle^3}{(1-|x|^2)^6}\right)  x_ix_jx_k\beta+\frac{3}{L^5}  \frac{1}{(1-|x|^2)^3}y_iy_jy_k\beta\\
&+\frac{3}{L^5} \frac{\langle x,y\rangle}{(1-|x|^2)^4}\left(y_iy_jx_k+y_jy_kx_k+y_ky_ix_j\right)\beta\\
&+\left(\frac{3}{L^5}  \frac{\langle x,y\rangle^2}{(1-|x|^2)^5}-\frac{1}{L^3}  \frac{1}{(1-|x|^2)^3}  \right)\left(y_ix_jx_k+y_jx_kx_k+y_kx_ix_j\right)\beta.
\end{align*}

By contracting the above equation by $\delta^{ij}$ and combining like terms, we get the following:

\begin{align*}
0&=G^h_{ijk}b_h\delta^{ij}\\&= \frac{1}{A^4L^3}\left((n+2) A^3L^2 -A^2 u^2+r^2A^2L^2-r^2\langle x,y\rangle^2-2\langle x,y\rangle^2\right)b_k\\
&+\frac{1}{A^5L^5}\big{(} -(n+4) A^3L^2 \beta -2 A \langle x,b\rangle L^2  \langle x,y\rangle +2A^2\beta u^2+6 A \beta\langle x,y\rangle^2 +r^2  \beta(3 \langle x,y\rangle^2-A^2L^2)   \big{)}y_k\\
&+\frac{1}{A^6L^5}\big{(} -(n+2) A^3L^2 \beta   \langle x,y\rangle +2 A^4L^4 \langle x,b\rangle   -   2 A^2L^2\langle x,b\rangle \langle x,y\rangle^2-2A^2\beta L^2  \langle x,y\rangle  \\
&+3 r^2  \beta  \langle x,y\rangle^3-   3 r^2 A^2L^2  \beta  \langle x,y\rangle+3 A^2\beta u^2      \langle x,y\rangle+  
6A \beta \langle x,y\rangle^3-2 A^3L^2 \beta \langle x,y\rangle\big{)}x_k,
\end{align*}
where we use the notations $$A=1-|x|^2, \quad r=|x|, \quad u=|y|, \quad b=(b_1,b_2,\cdots,b_n).$$

Again, by contracting by $x^k$, we have
\begin{align*}
0&=G^h_{ijk}b_h\delta^{ij}x^k\\
&= \frac{1}{A^4L^3}\left((n+2) A^3L^2 -A^2 u^2+r^2A^2L^2-r^2\langle x,y\rangle^2-2\langle x,y\rangle^2\right)\langle x,b\rangle\\
&+\frac{1}{A^5L^5}\big{(} -(n+4) A^3L^2 \beta -2 A \langle x,b\rangle L^2  \langle x,y\rangle +2A^2\beta u^2+6 A \beta\langle x,y\rangle^2 +r^2  \beta(3 \langle x,y\rangle^2-A^2L^2)   \big{)}\langle x,y\rangle\\
&+\frac{1}{A^6L^5}\big{(} -(n+2) A^3L^2 \beta   \langle x,y\rangle +2 A^4L^4 \langle x,b\rangle   -   2 A^2L^2\langle x,b\rangle \langle x,y\rangle^2-2A^2\beta L^2  \langle x,y\rangle  \\
&+3 r^2  \beta  \langle x,y\rangle^3-   3 r^2 A^2L^2  \beta  \langle x,y\rangle+3 A^2\beta u^2      \langle x,y\rangle+  
6A \beta \langle x,y\rangle^3-2 A^3L^2 \beta \langle x,y\rangle\big{)}r^2,
\end{align*}
Then, we have the following algebraic equation 
\begin{equation}
\label{Eq:Polynomial}
 A_1 \langle x,b\rangle  u^4+ (A_2\beta    + A_3 \langle x,b\rangle  \langle x,y\rangle)  \langle x,y\rangle u^2+A_4 \langle x,b\rangle \langle x,y\rangle^4+A_5 \beta  \langle x,y\rangle^3=0,
\end{equation}
 
where
\begin{align*}
 A_1:&=((n-2)r^6-3(n-1)r^4+3n r^2-(n+1)),\\
 A_2:&=(-3r^6+(n+3)r^4-2(n+1) r^2+n+2)   ,\\
 A_3:&=(-2nr^4+(4n-1)r^2-(2n-1) )   ,\\
 A_4:&=( (n+2)r^2-(n-2) )   ,\\
 A_5:&=( 2r^4-(n-2)r^2+n-2 )   .\\
\end{align*}
The equation \eqref{Eq:Polynomial} represents a polynomial of degree $4$ in the $y^i$ . Since this polynomial holds for all values of the $y^i$ , all coefficients must vanish, including the coefficients of $y_1^4, y_2^4, \ldots, y_n^4$, which are listed respectively by:
\begin{align*}
 A_1 \langle x,b\rangle+ A_2 b_1 x_1+A_3   x_1^2 \langle x,b\rangle+A_4 x_1^4 \langle x,b\rangle +A_5b_1x_1^3=&0 \\
 A_1 \langle x,b\rangle+ A_2 b_2 x_2+A_3   x_2^2 \langle x,b\rangle+A_4 x_2^4 \langle x,b\rangle +A_5b_2x_2^3=&0 \\
  \cdots  \qquad  \cdots  \qquad \cdots\qquad    \cdots\qquad \cdots   \qquad    \cdots  \qquad &\\
 A_1 \langle x,b\rangle+ A_2 b_n x_n+A_3   x_n^2 \langle x,b\rangle+A_4 x_n^4 \langle x,b\rangle +A_5b_nx_n^3=&0 
\end{align*}
Summing the above $n$ equations, we obtain:
 $$n A_1 \langle x,b\rangle +A_2 \langle x,b\rangle+A_3 r^2 \langle x,b\rangle  + A_4 \rho^2 \langle x,b\rangle +A_5   \langle x+\xi,b\rangle =0,$$
 where $\xi=(x_1^2, x_2^2, \cdots, x_n^2)$ and $\rho=|\xi|$. Therefore, we can express the above equation as follows
  $$\langle( n A_1   +A_2  +A_3 r^2  + A_4 \rho^2   +A_5    ) x+A_5 \xi,b\rangle =0.$$
Since the above equation holds for all $x^i$, and  the case where $A_5=0$ and $n A_1   +A_2  +A_3 r^2  + A_4 \rho^2   +A_5=0 $ provides polynomial in $x^i$ that cannot hold for all  $x^i$,       we conclude that $b_i=0$, implying that $\beta=0$. This means that there exists no non-zero parallel 1-form, hence the proof is completed.
\end{proof}

\section{Parallel 1-forms on spherically symmetric metrics}

Spherically symmetric Finsler metrics are important in Finsler geometry and physics. We're looking at whether they can have parallel 1-forms. A  Finsler structure  $F$ on $\mathbb{B}^n(r_0) \subset \mathbb{R}^n$ is spherically symmetric when it has the form:
 $$F(x,y)=u\phi\left(r,s\right),$$
where $r=|x|$, $u=|y|$,  $s=\frac{\langle x, y \rangle}{|y|}$, and   $\phi:[0,r_0)\times\mathbb{R}^n\to \mathbb{R}$.

We  lower  the indices  in $y^i$ and $x^i$  using  the Kronecker  delta $\delta_{ij}$, which are the components of the metric tensor attached  to  the Euclidean norm, as follows:
$$y_i:=\delta_{ij} y^j, \quad x_i:=\delta_{ij} x^j.$$
That is, $y_i$ and $x_i$ coincide with  $y^i$ and $x^i$ respectively. Moreover, we have
$$y^iy_i=u^2, \quad x^ix_i=r^2, \quad y^ix_i=x^iy_i=\langle x,y \rangle.$$
For further details, refer to, for example, \cite{Elgendi-SSM,Elgendi-Note,Guo-Mo}.

The following derivatives will prove useful in subsequent calculations.
\begin{equation}
\label{Eq:derivatives}
    \begin{split}
        \frac{\partial r}{\partial x^k}= \frac{1}{r}x_k, \quad  \frac{\partial u}{\partial y^k}= \frac{1}{u}y_k,  \quad  \frac{\partial u}{\partial x^k}= 0,         \quad  \frac{\partial r}{\partial y^k}= 0,  \quad
         \frac{\partial s}{\partial x^k}= \frac{1}{u}y_k, \quad  \frac{\partial s}{\partial y^k}= \frac{1}{u}(x_k-\frac{s}{u}y_k).   \\
    \end{split}
\end{equation}

The  coefficients  $G^i$  of the geodesic spray  of   the  Finsler metric  $F=u\phi(r,s)$ are expressed as follows:  
\begin{equation}\label{G}
  G^i=uPy^i+u^2 Qx^i,
\end{equation}
where   $P$ and $Q$ are given  by
\begin{equation}\label{P,Q}
  P:=-\frac{Q}{\phi}(s\phi +(r^2-s^2)\phi_s)+\frac{1}{2r\phi}(s\phi_r+r\phi_s),
\end{equation}
\begin{equation}\label{Q}
 Q:=\frac{1}{2 r}\frac{ -\phi_r+s\phi_{rs}+r\phi_{ss}}{\phi-s\phi_s+(r^2-s^2)\phi_{ss}}, 
\end{equation}
where the subscripts $s$ (resp. $r$) denote  the derivative   with respect to $s$ (resp. $r$).

The components  $G^i_j$ of the nonlinear connection   of  $F$ are  

\begin{equation}
\label{Eq:G^i_j}
G^i_j=u P \delta^i_j+P_sx_jy^i+  \frac{1}{u}(P-sP_s) \ y_jy^i+u Q_s x^ix_j+(2Q-sQ_s)  x^iy_j.
\end{equation}

 Computing the functions $P$ and $Q$ enables us to determine the geodesic sprays generated by the coefficients $G^i$ given in \eqref{G} for the spherically symmetric   metric $F = u \phi$. Conversely, the inverse problem involves recovering the Finsler metric from specified functions $P$ and $Q$. For solving the inverse problem, we have the following  lemma which obtained in \cite{Elgendi-SSM}.
 
\begin{lemma}\cite{Elgendi-SSM}
Given arbitrary functions $P(r, s)$ and $Q(r, s)$, let $F = u\phi(r, s)$ be a Finsler function whose geodesic spray is determined by $P$ and $Q$. Then, the function $\phi$ must satisfy the following two PDEs:

\begin{equation}
\label{Comp_C_C_2}
    \begin{split}
       (1+sP-(r^2-s^2)(2Q-sQ_s))\phi_s+(s P_s-2P-s(2Q-sQ_s))\phi &=0,   \\
         \frac{1}{r}\phi_r-(P+Q_s(r^2-s^2))\phi_s-(P_s+sQ_s) \phi &=0.
    \end{split}
\end{equation}
\end{lemma}
We shall refer to the aforementioned conditions as the 'metrizability conditions', as they directly determine the metrizability of the spray.

\begin{proposition} Assume that  $F=u\phi(r,s)$  is    spherically symmetric. Then, $F$ admits a parallel 1-form $\beta=b_iy^i$ if and only if  its   geodesic spray   is determined   by  the functions 
 \begin{equation}
\label{Eq:Metrizabilty_S}
P=P(r,s), \quad Q=\frac{ s^2 f'(r)}{2r^3f(r)}-\frac{sP }{r^2}+\frac{1}{2r^2}
\end{equation}
where  $b_i=f(r) x_i$,  $f':=\frac{d f}{d r}$, and     $f(r)$ is a  smooth  function of $r$.
\end{proposition}
\begin{proof}
Let $F$ be given by   $F=u\phi(r,s)$ such that $F$    admits   a parallel  1-form   $\beta$. Then $\beta$ can be written in the form $\beta=b_i(r)y^i$. 
Since $\beta$ is parallel, then $b_i$ must be gradient, that is, there is a function $h(r)$ such that $b_i=\frac{\partial h}{\partial x^i}$. Therefore, we can write
$$b_i=\frac{\partial h}{\partial x^i}=\frac{d h}{d r } \frac{\partial r}{\partial x^i}=\frac{1}{r}\frac{d h}{d r }x_i=f(r)x_i$$
where we set $f(r)=\frac{1}{r}\frac{d h}{d r}$. Moreover, we have
$$\beta=b_iy^i=f(r)x_iy^i=f(r)\langle x,y\rangle=f(r) su.$$
Now,   the condition $d_h\beta=0$ implies
$$\delta_i\beta=\pa_j (b_iy^i)-G^i_j \paa_i \beta=0.$$
Using   \eqref{Eq:derivatives} and plugging  \eqref{Eq:G^i_j} into the above equation, we have 
$$u\left(\frac{s f' }{r}-sf P_s-fP-fr^2Q_s \right)x_i+(f-fsP+fs^2P_s-2fr^2Q+fsr^2Q_s)y_i=0,$$
which implies the equations
\begin{equation}
\label{Eq:SSS_1}
\frac{sf'}{r}-sf P_s-fP-fr^2Q_s=0,
\end{equation}
\begin{equation}
\label{Eq:SSS_2}
f-fsP+fs^2P_s-2fr^2Q+fsr^2Q_s=0.
\end{equation} 
Adding  \eqref{Eq:SSS_2} to the mulitple of \eqref{Eq:SSS_1} by $s$, we obtain
\begin{equation}
\label{Eq:SSS_3}
\frac{s^2f'}{r}+f-2fsP-2fr^2Q=0.
\end{equation}
By making use of \eqref{Eq:SSS_3}, we have
$$Q=\frac{ s^2 f'}{2r^3f}-\frac{sP }{r^2}+\frac{1}{2r^2}$$
where $P$ is arbitrary  and this completes the proof.
\end{proof}

\begin{theorem}\label{Th:SSM}
The spray determined by the functions $P$ and $Q$ given in \eqref{Eq:Metrizabilty_S} is only Riemann  metrizable. Therefore,  the non-Riemannian spherically symetric metrics do not admit   parallel 1-forms.
\end{theorem}

\begin{proof}
By substituting by the formula   \eqref{Eq:Metrizabilty_S} of   $P$ and $Q$  into the metrizability conditions \eqref{Comp_C_C_2}, we have
\begin{equation}
\label{Eq_Phi_s}
(1+sP-(r^2-s^2)(2Q-sQ_s))\frac{\phi_s}{\phi}+ s P_s-2P-s(2Q-sQ_s) =0,
\end{equation}
Substituting   the expressions of $P$ and $Q$ into the above formula, we have
$$  \left(-sP_s(r^2-s^2)+P(2r^2-s^2)+s\right)\left(s \frac{\phi_s}{\phi}-1\right)=0.$$
Now, we have two cases either   $\frac{\phi_s}{\phi}=\frac{1}{s}$ or 
$$-sP_s(r^2-s^2)+P(2r^2-s^2)+s=0.$$
If $\frac{\phi_s}{\phi}=\frac{1}{s}$, then we get $\phi=c(r) s$.
Hence, the Finsler metric function $F$ is given by
$$F=u\phi=c(r) su=c(r) \langle x,y\rangle.$$
As the inner product $\langle x,y \rangle$ is linear in $y$, so is the formula for $F$. This implies that the metric tensor $g_{ij}$ is degenerate,  consequently,  the spray is non-Finsler metrizable.

Now, assume that $$-sP_s(r^2-s^2)+P(2r^2-s^2)+s=0.$$
Rewriting the above equation as follows
$$ P_s- \frac{2r^2-s^2}{s(r^2-s^2)}P =\frac{1}{r^2-s^2}.$$
Which can be seen as a linear differential equation with the solution
$$P=-\frac{s}{r^2}+\frac{c_1(r) s^2}{\sqrt{r^2-s^2}}$$
where $c_1(r)$ is a   function of $r$.

To determine  the function $\phi$, let's rewrite \eqref{P,Q} as follows
\begin{equation}
\label{Eq:New_P}
(2(r^2-s^2) Q-1) r\phi_s -s \phi_r +2r (P+sQ)\phi=0.
\end{equation}
 Differentiating both sides of the preceding equation with respect to $s$ yield
$$(2(r^2-s^2) Q-1) r\phi_{ss}-s\phi{rs}-\phi_r+2(P+(r^2-s^2) Q_s)r\phi_s+2r(P_s+Q+sQ_s)\phi=0.$$
Rewriting \eqref{Q} as follows
$$(2(r^2-s^2) Q-1) r\phi_{ss}-s \phi_{rs}+\phi_r+2 rQ \phi -2r s Q   \phi_s=0.$$
Subtracting the above two equations, we get
\begin{equation}
\label{Eq:Revision_3}
\left(P+sQ+(r^2-s^2)Q_s\right)r\phi_s-\phi_r+r(P_s+sQ_s)\phi=0
\end{equation}
By substituting by $\phi_r$ form \eqref{Eq:Revision_3} into \eqref{Eq:New_P}, we have
$$\left(2(r^2-s^2) Q-1-s(P+sQ+(r^2-s^2)Q_s\right)r\phi_s+\left(2rP+2rs Q-rs(P_s+sQ_s)\right)\phi=0.$$
Substituting   $P$ and $Q$ which are given in \eqref{Eq:Metrizabilty_S}, we get 
$$ (-2 f c_1 r^2s^3+(rs^2f'+r^2f+2s^2f)\sqrt{r^2-s^2}) \phi_s =0.$$
 If $ -2 f c_1 r^2s^3+(rs^2f'+r^2f+2s^2f)\sqrt{r^2-s^2}f =0$, then we have
$$-2 f c_1 r^2s^3=-(rs^2f'+r^2f+2s^2f)\sqrt{r^2-s^2}.$$
Squaring both sides of the preceding equation and collecting like terms yields
$$(4 c_1r^4f^2+r^2f'^2+4rff'+4f^2)s^6-r^3(rf'^2+2ff')s^4-4rfr^4(2rf'+3f)s^2-r^6f^2=0.$$
Which is a polynomial of degree $6$ in $s$ and satisfied for all values of $s$. That is,  we get
 $f=0$ and hence  $\beta=0$ which is the trivial case. If $\phi_s=0$, then $\phi=\psi(r)$. Thus, the Finsler function $F$ is given by
$$F=u\phi=\psi(r) u$$
which is Riemannian.  
\end{proof}

\section*{Conclusion}
We end this work with the following observations:

$\bullet$ As discussed at the beginning of Section 3, the presence of a parallel 1-form  $\beta $ on a Finsler space $(M, F)$  is  characterized by the   system:
$$d_h\beta=0,  \quad    d_\C\beta=\beta$$
As well as the compatibility condition $d_R\beta=0$. 
 Alternatively, we can consider the parallel property with respect to the horizontal covariant derivative of the Berwald connection. It's worth noting that we obtain the same concept even if we use the horizontal covariant derivative with respect to the Cartan, Chern, or Hashiguchi connections.  This is because the fundamental four connections have the same components $R^h_{ij}$ of  h(v)-torsion, see \cite[Table 1]{gomaa}.
 
\medskip 
 
$\bullet$ The class of $(\alpha,\beta)$-metrics comprises a Riemannian structure  $\alpha$ and a 1-form $\beta$. If $\beta$ is parallel with respect to $\alpha$ (the Levi-Civita connection), then any $(\alpha,\beta)$-metric $F$ inherits certain geometric properties from the background metric $\alpha$. For example, $\alpha$ and $F $ share the same geodesic spray and connection coefficients. Therefore, one can conclude that a parallel 1-form $\beta$ allows for the construction of numerous Finsler metrics that share the same geodesic spray.

\medskip  

$\bullet$ The presence of a parallel 1-form $\beta$ on a Riemannian space $(M,\alpha)$ not only preserves certain geometric properties but also has its own impact on the Finsler space constructed by any $(\alpha,\beta)$-metrics. For instance, Landsberg manifolds admitting a parallel 1-form exhibit a mean Berwald curvature of rank at most $n-2$. As a result, Landsberg surfaces with parallel 1-forms are necessarily Berwaldian. Additionally, the metrizability freedom of the geodesic spray for Landsberg metrics with parallel 1-forms is at least $2$.

\medskip 

$\bullet$ We figure out that some special Finsler metrics do not admit a parallel 1-form. Specifically, no parallel 1-form is admitted for any Finsler metrics of non-vanishing scalar curvature, among them the projectively flat metrics with non-vanishing scalar curvature. Furthermore,  neither the general Berwald's metric nor  non-Riemannian  spherically symmetric metrics admit a parallel 1-form. Consequently, we observe that certain $(\alpha,\beta)$-metrics and generalized $(\alpha,\beta)$-metrics do not admit parallel 1-forms.

 
 \providecommand{\bysame}{\leavevmode\hbox
to3em{\hrulefill}\thinspace}
\providecommand{\MR}{\relax\ifhmode\unskip\space\fi MR }
\providecommand{\MRhref}[2]{%
  \href{http://www.ams.org/mathscinet-getitem?mr=#1}{#2}
} \providecommand{\href}[2]{#2}


\begin{thebibliography}{10}


\bibitem{Berwald}
L. Berwald, \emph{$\ddot{U}$berdien-dimensionalen Geometrien konstanter Kr$\ddot{u}$mmung, in denen die Geraden die K$\ddot{u}$rzesten sind}, Math. Z., \textbf{30} (1929), 449--469.

 

\bibitem{Shen-book}
S. S. Chern and Z. Shen, \emph{Riemann-Finsler Geometry}, World Scientific Publishers 2004.


\bibitem{crampin2}
M. Crampin, \emph{A condition for a Landsberg space to be Berwaldian}, Publ. Math. Debrecen, \textbf{93} (2018), 143--155.



\bibitem{Elgendi-SSM}
 {S. G. Elgendi}, \textit{ On the classification of Landsberg spherically symmetric
Finsler metrics, } Int. J. Geom. Methods Mod. Phys., \textbf{18} (2021), 2150232 (26 pages).



\bibitem{Elgendi-Note}
 {S. G. Elgendi}, \textit{ A not on \lq\lq On the classification of Landsberg spherically symmetric
Finsler metrics\rq\rq}, Int. J. Geom. Methods Mod. Phys., \textbf{20}, 6 (2023), 2350096 (12 pages).



\bibitem{Mu-Elgendi}
S. G. Elgendi and Z. Muzsnay, \emph{Freedom of h(2)-variationality and metrizability of sprays}, Differ. Geom.  Appl., \textbf{54} A (2017), 194--207.

 \bibitem{Salah_ams_1}
S. G. Elgendi and Z. Muzsnay, \emph{The geometry of geodesic invariant functions and applications to Landsberg surfaces}, AIMS Mathematics, \textbf{9}   (2024), 23617-23631.




\bibitem{Guo-Mo}

E. Guo and X. Mo, \emph{The geometry of spherically symmetric Finsler manifolds}, Springer, 2018.




\bibitem{Kozma-Elgendi}
L. Kozma  and S. G. Elgendi, \emph{On the existence of parallel one forms},  Int. J.   Geom. Methods   Mod. Phys., \textbf{ 20},7 (2023), 2350118 (16 pages).






\bibitem{MZ_ELQ}
Z. Muzsnay, \emph{ The Euler-Lagrange PDE and Finsler metrizability}, Houston J. Math., \textbf{32} (2006), 79--98.


\bibitem{Percell}
P. Percell,  \emph{Parallel vector fields on manifolds with boundary},  J. Differ. Geom.,  \textbf{16}, (1981), 101--104.

 \bibitem{Grifone_1972} \emph{J.~Grifone, \emph{Structure presque-tangente et
      connexions. {I}.}  Ann. Inst. Fourier (Grenoble), \textbf{22} (1)  (1972),  287–334,
    .}


\bibitem{Shibata}
C. Shibata, \emph{ On invariant tensors of $\beta$-changes of Finsler metrics}, J. Math. Kyoto Univ.  \textbf{24}, (1984),
  163--188.
 



\bibitem{Mahara}  I. Mahara, \emph{Parallel vector fields and Einstein equations of gravity}, Rawanda J.  Series C,  \textbf{20} (2011),
106--114.

\bibitem{shen-book1}
Z. Shen, \emph{ Differential geometry of spray and Finsler spaces}, Springer, 2001.





 \bibitem{Shen-paper}
 Z. Shen, \emph{Projectively flat Finsler metrics of constant flag curvature}, Trans. of Amer. Math. Soc., \textbf{255} (2003), 1713-1728

 
\bibitem{gomaa}
Nabil L. Youssef, S. H. Abed and S. G. Elgendi, \emph{Generalized
$\beta$-conformal change of Finsler metrics},
   Int. J. Geom. Meth. Mod. Phys., \textbf{7}, 4 (2010), 565--582.
ArXiv Number: 0906.5369 [math.DG].


\end{thebibliography}
\end{document}